\documentclass[10pt]{amsart}
\usepackage{amscd}
\usepackage{amsmath}
\usepackage{amssymb}
\usepackage{latexsym}
\usepackage{lscape}
\usepackage{xypic}
\usepackage{hyperref} 
\usepackage{color}

\newtheorem{thm}{Theorem}[section]

\newtheorem{lem}[thm]{Lemma}

\newtheorem{lem-def}[thm]{Lemma-Definition}
\newtheorem{cor}[thm]{Corollary}
\theoremstyle{remark}

\newtheorem{rmk}{Remark}[section]
\theoremstyle{definition}
\newtheorem{dfn}{Definition}[section]

\numberwithin{equation}{section}

\newcommand{\quash}[1]{}  
\newcommand{\nc}{\newcommand}
\nc{\on}{\operatorname}

\newcommand{\bQ}{{\mathbb Q}}

\newcommand{\bZ}{{\mathbb Z}}
\newcommand{\calA}{{\mathcal A}}
\newcommand{\calB}{{\mathcal B}}
\newcommand{\calC}{{\mathcal C}}

\newcommand{\calF}{{\mathcal F}}
\newcommand{\calG}{{\mathcal G}}

\newcommand{\calU}{{\mathcal U}}

\newcommand{\calY}{{\mathcal Y}}

\nc{\al}{{\alpha}} \nc{\be}{{\beta}} \nc{\ga}{{\gamma}}
\nc{\ve}{{\varepsilon}} \nc{\Ga}{{\Gamma}} \nc{\la}{{\lambda}}
\nc{\La}{{\Lambda}}

\nc{\ad}{{\on{ad}}}

\nc{\aff}{{\on{aff}}}
\nc{\Aff}{{\mathbf{Aff}}}

\nc{\Bun}{{\on{Bun}}}

\nc{\der}{{\on{der}}}

\nc{\diag}{{\on{diag}}}

\nc{\Fl}{{\calF\ell}}

\nc{\Hol}{{\on{Hol}}}

\nc{\Id}{{\on{Id}}}

\nc{\Ind}{{\on{Ind}}}

\nc{\res}{{\on{res}}}

\nc{\tr}{{\on{tr}}}

\nc{\GSp}{{\on{GSp}}} \nc{\GU}{{\on{GU}}} \nc{\SL}{{\on{SL}}}
\nc{\SU}{{\on{SU}}} \nc{\SO}{{\on{SO}}}

\nc{\four}{{\calF our}}

\def\question#1{{}}

\author{An Huang, Shing-Tung Yau}
\address{Department of Mathematics\\ Harvard University\\Cambridge, MA 02138.}

\title{On cohomology theory of (di)graphs}
\date{Aug 29, 2014}

\begin{document}
\maketitle
\begin{abstract}
To a digraph with a choice of certain integral basis, we construct a CW complex, whose integral singular cohomology is canonically isomorphic to the path cohomology of the digraph as introduced in \cite{GLMY}. The homotopy type of the CW complex turns out to be independent of the choice of basis. After a very brief discussion of functoriality, this construction immediately implies some of the expected but perhaps combinatorially subtle properties of the digraph cohomology and homotopy proved very recently \cite{GLMY2}. Furthermore, one gets a very simple expected formula for the cup product of forms on the digraph. On the other hand, we present an approach of using sheaf theory to reformulate (di)graph cohomologies. The investigation of the path cohomology from this framework, leads to a subtle version of Poincare lemma for digraphs, which follows from the construction of the CW complex.
\end{abstract}

\tableofcontents
\baselineskip=16pt plus 1pt minus 1pt
\parskip=\baselineskip

\pagenumbering{arabic}
\addtocounter{page}{0}

\section{Introduction}\label{Intro}

In the past few years, there are rapidly increasing interests of developing geometric concepts in the context of graphs, besides spectral graph theory. See e.g. \cite{K0} for a short exposition. In particular, there exist several attempts to define the homology and cohomology of (di)graphs, e.g. via cliques \cite{BSS}, or via path algebra \cite{GLMY}\cite{DH}. 

Our first purpose of this paper is to try to better understand the path cohomology of digraphs \cite{GLMY}. This is an interesting theory which is expected to play the role of singular cohomology or in some nice cases De Rham cohomology for digraphs. \cite{GLMY} discusses some of its nice but perhaps subtle properties, and furthermore even nicer and expected basic properties regarding homotopy are proved in \cite{GLMY2}, by applying ideas of traditional algebraic topology to digraphs. In this paper, we construct a CW complex from a digraph with a given choice of certain integral basis, whose integral singular cohomology is canonically isomorphic to the path cohomology of digraphs. We will see that this construction gives rise to a functor from the category of digraphs, to a skeleton of the homotopy category of CW complexes, preserving products. Some immediate consequences of the existence of this functor include that the path cohomology of digraphs is homotopy invariant, that the Kunneth formula holds, and that there exists a functorial cup product on the path cohomology that can be lifted to the level of forms, as \cite{GLMY2}\cite{GLMY} proved. Furthermore, one can then define arbitrary higher homotopy groups for a digraph, to be that of the CW complex, where it is also considered in \cite{GLMY2} but in a possibly slightly different way. In addition, we get a very simple formula for the cup product of forms on the digraph, which is actually implicitly contained in \cite{GLMY2}, but here we understand the formula in a more geometric way. We hope these results are the beginning of a systematic investigation of this construction, which we hope to provide a useful bridge between digraph theory and topology. This construction may be thought of as a generalization of associating a simplicial complex to a graph, but is much more subtle, and has better categorial behavior. Intuitively, it may be viewed as, in some sense, a reversed construction to a particular generalized concept of triangulation of a manifold, which we hope to investigate in future works.

As the combinatorial Laplacian is a central object in (di)graph theory, one clear motivation for developing (di)graph cohomology theories is, in particular, to get interesting (di)graph analogues of Laplacian acting on differential forms, as a foundation for later developments. Some known cohomology theories of (di)graphs are similar to the conventional cohomologies for topological spaces, but at the same time also seem to exhibit some different and perhaps puzzling features at first glance. The CW complex can help to understand this issue better, however we also hope to understand it from different points of view. Also one should ask how these different cohomology theories may be related or treatly in a uniform way. Our second purpose of this paper, starts from section \ref{clique},
is to use sheaf theory to study (di)graph cohomology theories, with the hope to treat different theories within a single framework. It turns out that there is a Poincare lemma for the path cohomology of digraphs, which follows from the construction of the CW complex mentioned in the previous paragraph. Our approach here is partly inspired by some recent study of topologies on a graph \cite{K}, and our motivation partly lies on the hope that the sheaf theory idea might eventually lead to a much-hoped cohomological proof of the Riemann-Roch theorem for graphs \cite{BN}.

{\it Acknowledgements.} The authors thank CASTS (Center of Advanced Study in Theoretical Sciences) of National Taiwan University, where most of the work was done during their visit. They also thank Fan Chung, Alexander Grigor'yan and Yong Lin for useful discussions.

\section{From digraph to CW complex}\label{CW}
In this section, we follow notations in \cite{GLMY}, with some modifications that we will mention. Let $G$ be a finite digraph. By a primitive allowed $k$-path, we mean an ordered sequence of vertexes $i_oi_1...i_k$, such that $i_si_{s+1}$ is a directed edge in $G$, for $s=0,1,...,k-1$. We say this primitive allowed path is regular, iff all these vertexes are different from each other. Note that this regularity condition is more restrictive than the one used in \cite{GLMY}. There are several reasons we prefer this regularity condition: e.g. with this new condition, the homology groups are now obviously bounded above, and Lefschetz fixed point theorem holds, \footnote{We will explain this briefly in section \ref{L}.} both of which are not true with the old regularity condition. On the other hand, we will make a try to relax our regularity condition at the end of this section, to extend the applicability of our construction. We Let $A_k(G)$ denote the space of regular allowed $k$-paths, which by definition, is the free $\bZ$-module generated by all regular primitive $k$-paths, and let $\Omega_k(G)$ denote the submodule of $\partial$-invariant regular allowed $k$-paths defined recursively, as in \cite{GLMY}: recall this means the subspace of $A_k(G)$ consisting of elements whose boundary is an element of $A_{k-1}(G)$. We also use $\Omega(G)$ to denote the direct sum of $\Omega_k(G)$ for all $k$. When no confusions arise, we omit $G$ and write $A_k$ and $\Omega_k$. We call $k$ the length of the path. Note that $A_k(G)=0$ when $k\geq |G|$.

For any $P=\sum_{k=1}^mc_kp_k\in \Omega_k(G)$, where $p_k,k=1,2,...,m$ are primitive regular allowed paths, we define $w(P)=\sum_{k=1}^m|c_k|$ to be the width of the path $P$. For each $p_k$, we define its support to be the subgraph it defines, namely, the minimal subgraph of $G$, such that $p_k$ is an allowed path in the subgraph. We define the support of $P$ to be the union of the support of each $p_k$ where $c_k$ is nonzero, and denote by $Supp(P)$. We say $P$ is minimal, iff there do not exist integers $d_k,k=1,2,...,m$, such that $|c_k-d_k|\leq |c_k|$ and $|d_k|\leq |c_k|$ for each $k=1,2,...,m$, and $P'=\sum_{k=1}^md_kp_k\in \Omega_k$, and $w(P')<w(P)$. In this definition, if such a $P'$ exists, we say that $P'$ is strictly smaller than $P$. Note that $Supp(P')\subset Supp(P)$, and we have also $P-P'\in\Omega_k$ is strictly smaller than $P$. Therefore, it is clear that, any element in $\Omega_k(G)$ is a linear combination of minimal elements.

\begin{lem} \label{SE}
Any minimal path is a linear combination of primitive paths with the same starting and ending vertexes.
\end{lem}
\begin{proof}
Given any two primitive paths with different starting vertexes, that both show up in a $\partial$-invariant path, if some of their boundary components cancel possibly through a string of other primitive paths, at certain step one has to change the starting vertex, therefore the cancellation is not useful in eliminating non $\partial$-invariant paths, in the sense that there exists a strictly smaller $\partial$-invariant path consisting of primitive paths all starting with the same vertex. The same argument applies to the ending vertex.
\end{proof}

Now we are going to construct cells from minimal paths, and a CW complex given a choice of integral basis of $\Omega(G)$ consisting of minimal paths, whose existence is a corollary of lemma \ref{1} below, which we will prove together with lemma \ref{min} simultaneously by induction.
\begin{lem}\label{min}
Any minimal path $P$, is a linear combination of primitive paths, with coefficients being either $1$ or $-1$.
\end{lem}
\quash{
\begin{proof}
If some primitive path $p$ occurs more than once in $P$, then there exists at least one component of the boundary of $p_0$, that is used in cancellations more than once, with two other distinct primitive paths $p_1$ and $p_2$. Then the product of signs of $p_1$ and $p_2$ equals $1$, but at the same time they differ by only one place, thus contradicting lemma \ref{fundamental}.
\end{proof}
}

\begin{lem}\label{1}
Any minimal integral relation among minimal paths of a fixed length, is of the form $\sum_{i=1}^m \lambda_iP_i=0$, where all the coefficients $\lambda_i$ are either 1 or -1. Here the definition of minimal integral relations are the same as that in the definition of minimal paths-- in an obvious sense that it can not be written as a sum of two strictly smaller relations.
\end{lem}
\quash{
\begin{proof}
We follow the following procedure to write $P$ as such an integral linear combination: at every step, we first identify the highest component of the path, then find a lowest order basis element with same highest component ( note that this is always possible as the proofs of previous lemmas show) and subtract it from the path. It is obvious that this procedure has to end after finitely many steps, and we are left with the desired linear combination. By lemma \ref{fundamental}, any pair of primitive paths that both appear in any single minimal path has to appear with the same sign configuration in any other minimal path. (the product of their signs stays the same) Therefore, suppose $p$ is the maximal order primitive component of $P$ that appears in the linear combination more than twice, then there has to exist more than one $b_i$ in the representation with nonzero coefficient, that has highest component higher than $p$, (contributing $p$ with the same signs in the expression) and at least one $b_j$ with nonzero coefficient, that has highest component equal to $p$. Then upon subtracting one such $b_j$ from one such $b_i$ eliminating $p$, and decompose the result as a sum of minimal elements, and then write the minimal element involving the maximal order component as sums of basis elements (whose orders are lower), one gets another basis element with strictly lower order, while having the same maximal component as that of the $b_i$ one has chosen, contradicting the procedure.
}

Note that lemma \ref{1} implies that any rational basis of $\Omega(G)$ consisting of minimal paths is an integral basis, thus it implies the following. 

\begin{cor} \label{basis}
There exists an integral basis of $\Omega(G)$ consisting of minimal paths.
\end{cor}

For path length $k=1$, both lemmas are obviously true, and furthermore one associates a $k$-cell to any minimal path of length $k$, by filling in a $(k-1)$-sphere, corresponding to the union of cells associated with boundary components of the path: meaning that the boundary of the path decomposes uniquely as a sum of smaller minimal paths of length $k-1$, each of which we have associated a cell, and the cell association commutes with the boundary operation. Now suppose all of these statements are true for path length up to $k-1$, and take $P$ to be a minimal path of length $k$. As $\partial P$ is a path, it can be decomposed into a sum of minimal paths of length $k-1$, where all the paths are smaller than or equal to $\partial P$, for which there are associated $k-1$ cells. The union of these cells, counting multiplicity, is a sum of closed manifolds, as $\partial\partial P=0$. Note: the reader can convince himself/herself that, each minimal path of length $k-2$ that shows up as a boundary component of a boundary component of $P$ appears even times as expected, and in particular there are no singularities on these manifolds. We construct a height function on it as follows:  By lemma \ref{SE}, the starting and ending vertexes of any minimal path are unique. First of all, there is a height function on edges, given by piecewise-linearly extending the integer valued length function defined on vertexes, given by the position it sits in a primitive path component-- note that this position number is the same for any primitive path one chooses, as a consequence of the obvious fact that any primitive path of maximal length in the support of a minimal path, must be a component of the minimal path. We proceed by extending the height function to disks and so on, as we can always extend the height function from a sphere to the ball it bounds. Take any of these closed manifold and call it $M$, we can make a small perturbation to make the height function become a Morse function on $M$. We single out a subset $E_1$ of the set of vertexes in the support of $P$, consisting of vertexes such that any path in the support of $P$ connecting the vertex to the ending vertex $E$ is of length 1. We define another subset $S_1$ in the symmetric way, with respect to the starting vertex. It is clear that the only possible critical points of this height function are the starting and ending vertexes, and vertexes in $E_1$ and $S_1$, as aside from them, there is always a direction in which the function is strictly monotonic. Now suppose a vertex $W$ in $E_1$ is a critical point, then $E$ can not lie on $M$. Take any primitive path component $p$ of $P$, whose support after truncating $E$ is in $M$, that goes through $W$, \footnote{Such a primitive path has to exist in the situation.}and let us write it as $p=SqWE$. Then $qW$ as a boundary component of $SqW$, has to be cancelled by a boundary component of a primitive path in a minimal path associated with $M$, which all are consisting of primitive paths of length $k-1$ that does not go through $E$. So the only such possible primitive path that has a boundary component cancelling it is itself with a different orientation, which is a contradiction. For the same reason, vertexes in $S_1$ can not actually be critical points. Therefore, the only possible critical points are $S$ and $E$, which implies that $M$ is a $(k-1)$-sphere, and $S$, $E$ are in its support. Now take all length $k$ primitive paths in the support of $M$, defined by the union of the support of $k-1$ minimal paths associated with $M$, with orientation determined by orientations of boundary components, we get a $\partial$-invariant $k$-path that is smaller than or equal to $P$: recall that any maximal length primitive path in the support of $P$ is a primitive component of $P$. Also note that for any primitive path of length $k$, all of its boundary components must have support in a single sphere, as otherwise there have to exist boundary components that does not belong to any of the spheres, which is impossible. On the other hand, any primitive $(k-1)$-path associated with $M$ must be a boundary component of a primitive $k$-path with support in $M$, as a consequence of the fact that any longest primitive path in the support of $M$ must have length $k$. So the path we just constructed has to be equal to $P$ as $P$ is minimal, and therefore $\partial P$ corresponds to a single $(k-1)$-sphere. This implies lemma \ref{min} for $P$, and that the decomposition of $\partial P$ in to a sum of minimal elements is unique. On the other hand, we can fill in the sphere to get a $k$-cell for $P$. This cell association clearly commutes with the operation of taking boundary, by construction.

Let us choose an integral basis for each $\Omega_j(G)$ consisting of minimal paths, for $j$ up to $k-1$, where lemma \ref{1} is true by inductive hypothesis. We now construct a $(k-1)$-skeleton together with some $k$-cells that we will later use in the induction, from $G$ with our choice of basis of $\Omega(G)$ up to length $k-1$. For this purpose, we need to possibly exclude cells associated with minimal paths that are not elements of the basis chosen, and some cells may need to be modified accordingly.

Again, for any vertex, one associates a zero-cell. For any edge, one associates a one-cell with boundary given by the boundary of the edge. 

Suppose again for all minimal paths in our basis of length up to $i-1$, ($i\leq k$) one has associated cells of the corresponding dimension, by filling in a sphere that is associated with the boundary of the path, so one has a CW complex with cell dimensions up to $i-1$. Now pick any designated minimal path $P$ of length $i$, it must have a single starting vertex S and a single ending vertex E by lemma \ref{SE}. Again minimality of $P$ and lemma \ref{min} implies that the boundary of $P$ can be decomposed uniquely into an integral linear combination of minimal paths, with all coefficients being 1 or -1. For any minimal path $P'$ of length $i-1$ that shows up in the linear combination, if it is in our chosen basis, we have already assigned a $(i-1)$-cell to it. Otherwise, it is a unique integral linear combination of basis elements, with coefficients being 1 or -1 by lemma \ref{1}.

If we union the cells in the previous paragraph associated with basis elements in the linear combination, one gets a manifold with  boundary being a $(i-2)$-sphere specified by the union of all $(i-2)$-dimensional cells associated with $\partial P'$. \footnote{Note that a choice of basis is important for this to be true.}For any such manifold, we can again construct a Morse height function by gluing together individual such functions on cells associated with each minimal element-- note that the height function is constructed in a way that enables one to glue. Then the same argument shows that it is a $(i-1)$-disk: e.g, one can attach another $(i-1)$-cell with the  $(i-2)$-sphere, to get a manifold without boundary, and then uses the same Morse theory argument.

Therefore, for each minimal path that shows up in the above decomposition of the boundary of $P$, there corresponds a piece of the already existing CW complex homeomorphic to a $(i-1)$-disk, with boundary as we described. So again the boundary of these disks cancel, and therefore the union of all of them is again a manifold of dimension $i-1$. Note: one has to show further that each $(i-1)$-cell associated with basis elements that shows up has multiplicity 1 (or -1), when taking all the $P'$ into account. This can be done by first restricting our attention to all the $P'$ that share the same starting and ending vertexes. Each of these $P'$ corresponds to a previously constructed $(i-1)$-cell. One sees that the union of these cells is homeomorphic to a $(i-1)$-disk, as a consequence of the fact that all these $(i-1)$-cells of different starting and ending vertexes union to form a manifold without singularity. Then, if any of the $(i-1)$-cell aformentioned has greater multiplicity, by an argument similar to that in the inductive proof of lemma \ref{1} below, a sphere must be present, resluting from gluing together cells associated with certain basis elements that show up, so one creats a nontrivial linear relation among basis elements, which is impossible. Now the same Morse height function argument shows that this manifold is homeomorphic to a $(i-1)$-sphere, and therefore one can fill it in with a $i$-cell.

The previous induction goes up to $i=k$. To continue, we have yet to finish our inductive proof of lemma \ref{1} for length $k$. Suppose we have a minimal integral relation among minimal paths of length $k$, then obviously we have unique starting and ending vertexes for all primitive paths involved in this relation. So any such integral relation gives rise to a geometric fact that, the union of all these cells that we have just constructed corresponding to the minimal paths that show up in the relation, counting multiplicity, is a sum of manifolds without boundary,\footnote{Again, a choice of basis that we have already done up to length $k-1$ is important for this to be true} For any such manifold, we can again construct a Morse height function by gluing together individual such functions on cells associated with each minimal element, then the same argument shows that it is a sphere, and therefore corresponds to a minimal relation as one easily convinces oneself. So any minimal relation corresponds to a single sphere. Lemma \ref{1} is thus evident for $k$-paths. Our induction is thus complete.

Therefore we can choose a basis for $\Omega_k(G)$ consisting of minimal paths, and this inductive procedure continues until one associates a cell to each basis element one has chosen, and therefore ends up with a $k$-skeleton. Now one can simply take $k$ to be the upper bound where $\Omega_k(G)$ is nonzero, and one ends up with a CW complex, associated with a choice of integral basis of $\Omega(G)$ consisting of minimal paths. It is evident from the construction that, the cell association still commutes with the boundary operator, and the integral singular cohomology of the CW complex is canonically isomorphic to the digraph path cohomology.

Our next step is to construct a homotopy between any such CW complexes. For this purpose, it suffices to show it for each $k$ step by step, where $k$ is the length of path, and the change of basis can be done step by step, where for each step, only basis regarding length $k$ change.

Next, we let $a_1,...,a_s$ be any other integral basis of $\Omega_k(G)$ consisting of minimal elements. Then the change of basis from $b_1,...,b_s$ to $a_1,...,a_s$ can be done in a sequence of $s$ steps, where each step can be expressed as the form $c_1,c_2,...,c_s\rightarrow d_1,c_2,...,c_s$, corresponds to a change of a single basis element from $c_1$ to $d_1$ corresponding to a minimal integral relation expressing $d_1$ as an integral linear combination involving $c_1$ of the basis elements $c_1,c_2,...,c_s$. By lemma \ref{1} and its proof, we see there is the following continuous map of topological spaces that we can define:

$c_1$ corresponds to a cell. Write it as the unique integral linear combination of $d_1,c_2,...,c_s$. we "collapse" this $c_1$ cell onto the union of cells corresponding to this integral linear combination, which can be viewed as a refinement of the $c_1$ cell prescribed by this linear relation. This procedure does not affect cells of strictly lower dimensions, and it is clear that this "collapsing" can be extended to a continuous map of the two CW complexes corresponding to these two different basis: i.e. one extends this map in an obvious way to higher dimensional cells. There is of course a continuous map in the reversed direction by collapsing from the second basis to the first basis. One checks directly that the composition of these two maps is homotopic to the identity map, basically by "slowly pulling the string back".\footnote{The reader can convince himself/herself easily through a 1-dimensional example.} Therefore, any such collapsing is a homotopy.

Therefore for each digraph $G$, one assigns a CW complex unique up to homotopy. We next show that this assignment is functorial: meaning that it defines a functor from the category of digraphs, where morphisms are defined in a particularly strict sense that we will explain below, to a skeleton of the homotopy category of CW complexes: For this homotopy category, we mean that the objects are CW complexes, while the morphisms are homotopy classes of continuous maps of topological spaces. On the other hand, a skeleton may not sound attractive, however, it can help to express things fast in a more formal way that is useful to deduce some expected properties of digraph cohomology quickly. We will not go any deep into these abstract nonsense in this paper, nor will we make serious effort to find the best way to abstractly formulate this association of CW complexes to a digraph with a choice of basis, as that may better be done later if it becomes necessary. Given any map from $G$ to another digraph $G_1$, which means that vertex maps to vertex, and directed edge maps to directed edge, that preserves the incidence relations among directed edges and vertexes. For our first discussion below, we do not allow different vertexes to map to the same vertex, and we only establish the functoriality below in this narrow sense. We will see the discussion can probably be extended in a larger cartegory, where functorially in a broader sense holds. It is clear that any minimal path is mapped to a linear combination of allowed $\partial$-invariant paths of the same length, which is a sum of minimal paths. Pick any integral basis of $\Omega(G)$ and $\Omega(G_1)$ consisting of minimal paths, we construct a continuous map from $T_G$ to $T_{G_1}$ inductively: first, vertexes and directed edges are mapped to their images. Now suppose cells corresponding to minimal paths of length strictly less than $k$ are mapped, then for any minimal path $P$ of length $k$ in the chosen basis, the image can be decomposed as a sum of minimal paths of $G_1$ again with coefficients being 1 or -1, which themselves then correspond to unions of $k$-cells in $T_{G_1}$ homeomorphic to $k$-disks with boundary corresponding to the boundary of the minimal paths, and furthermore any basis element that appears is with multiplicity 1 or -1. Thus one can homeomorphically map the $k$-cell associated to $P$, to the union of these $k$-cells, in terms of a refinement (subdivision) of the cell, which gives the desired map inductively. Note that this procedure does not affect maps of cells of strictly lower dimensions that are already defined. It is then routine to check the functorial properties, as refinements compose in a desired way.

When directed edges are allowed to collapse, and in particular different vertexes are allowed to map to the same vertex, a digraph may be mapped to a multidigraph, which means multi-edges with arbitrary orientations and self-loops are allowed.\footnote{It can also happen that a digraph still maps to a digraph, but our previous discussion may encounter problems of degeneration.} In the larger category of multidigraphs, a morphism is defined to be a map that takes vertex to vertex, and directed edge to directed edge, that preserves the incidence relations among directed edges and vertexes. No more restrictions will be put. To extend our discussion to this larger category, we need to relax our definition of a path and the regularity condition, in a precise way that we allow paths that result from various kinds of degenerations. We exhibit in the following a candidate choice of such definitions.

\begin{dfn}
A virtual primitive path is an ordered string of vertexes $V_0V_1...V_s$, together with the following data: for any pair of consequtive vertexes $V_k,V_{k+1}$ in the string (k=0,1,...,s-1), either one specifies a directed edge connecting them, or  $V_kV_{k+1}$ is not a directed edge, and furthermore, for any consequtive pairs of vertexes in any string of vertexes that appear as a (formal) component of $\partial(P)$, or boundary components of boundary components and so on, one either specifies a directed edge connecting them, or there is no directed edge between them.  These specifications have to be done in a way compatible with all incidence relations among paths.
\end{dfn}

\begin{dfn}
A primitive path is a virtual primitive path, such that for any pair of consequtive vertexes $V_k,V_{k+1}$ in the string (k=0,1,...,s-1), either we specified a directed edge connecting them, or $V_kV_{k+1}$ is not a directed edge, but $V_k=V_{k+1}$.
\end{dfn}

\begin{dfn}
If a primitive path satisfies the additional condition that, any directed edge appears at most once as a segment of the path, and any virtual primitive path of length one less appears at most once in the formal components of $\partial(P)$ before any cancellation, then we call it a regular primitive path. 
\end{dfn}

\begin{dfn}
A path is an integral linear combination of regular primitive paths, and a $\partial$-invariant path is a path, whose $\partial$ is a linear combination of primitive paths. The space of $\partial$-invariant paths is denoted by $\Omega$. 
\end{dfn}

\begin{rmk}
This definition reflects the fact that, it is possible that, some boundary components of a cell collapse, while the cell itself stays a cell. So we do not require all boundary components to be regular.
\end{rmk}

From these definitions or perhaps some variants of them, we expect that the construction of the CW complex generalizes to multidigraphs, and functoriality holds in the broad sense stated. One needs to define the cohomology with a little more care similar to what is done in \cite{GLMY}, to account for the new regularity condition. We leave the details of this to a future writing. The cells in this more general setting, should all be regared as obtained from various contractions from the cells in the old setting.  For functoriality, given two multidigraphs $G$, $G_1$, a morphism between them, and a minimal path $P$ in $G$, one in general may need to contract the cells associated with $P$ in the way prescribed by the digraph morphism, and then do the map described above to match the choice of integral basis of the second multidigraph. Note that a cell may be mapped to lower dimensional cells in general.


\begin{rmk}
Note that, for the category of digraphs with morphism defined in our narrow sense, the resulting CW complex has the property that any attaching map is an obvious homeomorphism. However, things will be more complicated in the bigger category of multidigraphs.
\end{rmk}

Coming back to digraphs, by \cite{GLMY} it is evident that, given integral basis of two digraphs, then their product is an integral basis of the product digraph, and furthermore taking boundary of products of paths satisfies the Leibniz rule, which implies that our association of a CW complex to a digraph preserves products.

\begin{rmk}
As we have seen, one can associate a cell to any minimal path, and thus actually construct a CW complex in a canonical way, from $G$ without a choice of integral basis as above, and the construction also probably have all these nice functorial properties. However, the cohomology of this new CW complex will get additional contributions from linear relations among minimal paths, which perhaps makes this construction less appealing.
\end{rmk}

\section{Some immediate consequences}
It then follows from simple abstract nonsense that, a homotopy between digraphs induces isomorphisms of cohomology groups, and that the Kunneth formula holds for digraph cohomology. Furthermore, one can define arbitrary higher homotopy groups of a digraph, in terms of that of the CW complex. On the other hand, the cohomology of digraphs becomes a functorial graded ring as that of the CW complex is such a graded ring under the cup product. It turns out that this product can be lifted to the level of forms, which are defined to be elements in $\Omega^k$, the dual of $\Omega_k$,  and the lift still respects associativity and the Lebniz rule, and is functorial. Most of these are first proved in \cite{GLMY}\cite{GLMY2}. We show below that a very simple formula exists for this lifted product\footnote{The formula is actually implicitly contained in \cite{GLMY2}, or should be at least expected in any case, but here we provide a more geometric understanding of it.}, which may be relevant e.g. in studying some gauge field theories on the digraph. 

One sees from the construction of the CW complex that, for any minimal path in the chosen basis $P=\sum_{k=1}^mc_kp_k\in \Omega_k(G)$, where $p_k,k=1,2,...,m$ are primitive regular allowed paths, there exists a unique subdivision of cells, given by connecting all unconnnected edges in every $p_k$ in the same direction of the path, so that each $p_k$ becomes a complete graph. After this subdivision, the cell associated with $P$ is divided into a sum of simplexes, each associated with a $p_k$ with the newly connected edges. One can do this subdivision to all cells associated with basis elements in a consistent way, and then the CW complex becomes a simplicial complex, whose simplicial cohomology is canonically isomorphic to the singular cohomology of the CW complex. The cup product in this simplicial complex has the well-known simple formula in terms of simplexes, which then translates into the corresponding formula for the cup product in the CW complex restricted to the cells we are considering. Unravelling the definitions, one sees that this restriction actually gives rise to the formula for the functorial cup product of forms on digraphs. Let $\alpha\in\Omega^p(G)$, and $\beta\in\Omega^q(G)$, and $k=p+q$. Suppose $p_k=V_0...V_{p+q}$. We let $p_k|_{0...p}$ and $p_k|_{p...p+q}$ denote the allowed paths $V_0...V_p$ and $V_p...V_{p+q}$, respectively, resulted from truncating $p_k$ in the way described. Then we have the formula for the cup product $\alpha\cup\beta$ on $P$ as follows:

\begin{equation}
\alpha\cup\beta(P)=\sum_{k=1}^mc_k\alpha(p_k|_{0...p})\beta(p_k|_{p...p+q})
\end{equation}

Note that, the above formula does not make sense in a first glance, as each individual truncation may not be in $\Omega$, however, the formula is understood in the sense that one needs to first merge together all terms with the same truncation in the argument\footnote{Namely, the same $\alpha(p_k|_{0...p})$ or $\beta(p_k|_{p...p+q})$.}, in the above sum. Then it is an easy exercise to show that it indeed makes sense after the merging. One sees also from this formula that it is independent of our choice of basis.

Here we also explain a few words regarding the homotopy invariance property: a homotopy of two maps of digraphs is defined in direct analogy with the corresponding concept in topology \cite{GLMY2}, and applying our functor, any such homotopy gives rise to a homotopy between two continuous maps of the CW complexes associated with the two digraphs, therefore inducing isomorphic maps on cohomology groups. Furthermore, since homotopy of digraphs becomes homotopy of the CW complexes, our functor provides a tool to study homotopy properties of digraphs, stronger than just the cohomology. 

Furthermore, we expect all these to generalize to multidigraphs (quivers), as the previous section briefly discussed.

\section{Clique cohomology}\label{clique}
From this section, we start to use sheaf theory to reformulate some know (di)graph cohomology theories. This is a preliminary work, and only some very basic things will be presented below. We first illustrate the ideas with the example of clique cohomology, and here we try to follow notations in \cite{K0}.

Let $G$ be a finite graph.\footnote{More generally, the following theory also works for an infinite graph all of whose vertexes have finite degree.} Let $\calG_k$ denote the set of all $K_{k+1}$ subgraphs of $G$, and $\calG=\cup_{k=0}^{\infty}\calG_k$. By a topology $T$ on a graph $G$, we actually mean a topology $T$ on the set $\calG$. Take any topology, one can consider the category of sheaves of abelian groups on $\calG$. Sheaf cohomology is well-defined, as any such category has enough injectives. However, it is crucial that one chooses a suitable topology for all applications that follow. To mimic the case of usual continuous geometry, here we consider the unit ball topology, which is defined by a topology subbasis as the set of all unit balls, whose definition we state below:\footnote{Note there probably exist other good choices of topology for our purpose here.}

For any vertex $v\in G$, we define its unit ball subgraph $B_v$, as the subgraph of $G$, generated by $v$ and all of its neighbors. In other words, it is the largest subgraph of $G$ containing only these vertexes. For each $B_v$, we canonically associate a subset $\calB_v$ of $\calG$ as follows: $x\in\calG$ lives in $\calB_v$ if and only if $x$ is a subgraph of $B_v$.

It is clear from definition, that these $\calB_v$ give a subbasis of topology. Note that $B_v$ is a cone, therefore one has $H^i(B_v)=0$, for all $i>0$, here $H^i$ is the graph cohomology functor defined by the clique complex. For any $x\in\calG$, let us denote by $\calU_x$ the smallest open subset containing $x$, which always exists as there are finitely many such open sets. By our choice of topology, $\calU_x$ corresponds uniquely to a subgraph $U_x$ in the same sense that $\calB_v$ corresponds to $B_v$: $y\in\calG$ lives in $\calU_x$ if and only if $y$ is a subgraph of $U_x$. We have the following characterization of $U_x$:

\begin{lem}\label{stalk}
$U_x$ is the intersection of maximal complete subgraphs containing $x$.
\end{lem}
\begin{proof}
Suppose a vertex $v$ does not belong to some maximal complete subgraph $K$ containing $x$. Then there exists a vertex $w$ in $K$, such that $w$ is not connected to $v$ by an edge. Then $B_w$ contains $U_x$, but on the other hand, $v$ is not in $B_w$. So $v$ is not in $U_x$, which implies that $U_x\subset K$, so $U_x$ is contained in the intersection of maximal complete subgraphs $\cap K$ containing $x$. Conversely, if $v_1$ is a vertex such that $B_{v_1}$ contains $x$, then the complete graph $K_1$ containing both $v_1$ and $x$ is a subgraph of $G$. Consider the maximal complete subgraph $K_2$ of $G$ containing $K_1$: we have that $K_2\subset B_{v_1}$ by the definition of the unit ball subgraph. So $\cap K\subset K_2\subset B_{v_1}$, which proves the inclusion in the other direction.
\end{proof}
As a consequence, we have
\begin{cor}\label{scontractible}
$U_x$ is a complete subgraph, and in particular, $H^i(U_x)=0$ for any $i>0$.
\end{cor}

Take $A$ to be any abelian group, next we construct a flasque resolution of the constant sheaf $\calA$ on $\calG$ with values in $A$.

Take any $\calU\subset\calG$ an open subset.  Let $\calU_k$ denote $\calG_k\cap \calU$. Define $C^k(\calU)$ to be the abelian group of continuous functions, from $\calU_k$ to $A$, where $\calU_k$ is equipped with the subset topology (which actually does not matter), and $A$ the trivial topology. It is easy to check from definition, that the assignment $\calU\rightarrow C^k(\calU)$ defines a flasque sheaf $\calC^k$ on $\calG$. The differential of the clique complex gives rise to a differential mapping $\calC^k$ to $\calC^{k+1}$, and making it into a complex of sheaves. Furthermore, any section of the constant sheaf $\calA$ on $\calU$ is a function that is constant on every connected component of $\calU$, thus can be mapped to a section of $C^0(\calU)$, by associating the vertexes in each connected component the corresponding constant value in $A$. We have the following

\begin{lem}\label{res}
$\calC^k$ gives a flasque resolution of the constant sheaf.
\end{lem}
\begin{proof}
The exactness at $\calA$ and $\calC^0$ is obvious. At general $\calC^k$, we look at each stalk. Unraveling the definition, the exactness after taking stalks reduces to corollary \ref{scontractible}.
\end{proof}
Taking global sections, we therefore have the following

\begin{thm}\label{iso1}
There is a canonical isomorphism $H^i(\calG,\calA)\cong H^i(G,A)$.
\end{thm}
where $H^i(G,A)$ denotes the graph cohomology defined by cliques, taking values in $A$.

Next, we consider \v{C}ech cohomology. Take a finite open cover $\calU_i,i=1,2,...,s$ of $\calG$, one forms the \v{C}ech complex for any sheaf $\calF$ of abelian groups. As will be expected, we have

\begin{lem}
For each $i$, there is a natural map $\check{\mathrm{H}}^i(\calG,\calF)\rightarrow H^i(G,\calF)$, functorial in $\calF$.
\end{lem}
\begin{proof}
See \cite{H}, III.4.
\end{proof}
Take $\calF=\calA$. Take a finite open cover $\calU_i,i=1,2,...,s$ of $\calG$ such that any intersection has trivial higher cohomology. We as usual have the following
\begin{thm}\label{iso2}
The natural map above gives an isomorphism $\check{\mathrm{H}}^i(\calG,\calA)\cong H^i(\calG,\calA)$.
\end{thm}
\begin{proof}
See proof of theorem 4.5 on page 222 of \cite{H}.

\end{proof}
\begin{rmk}
Note that such an open covering always exists, and \ref{stalk} provides a canonical one as such, by \ref{scontractible} and \ref{iso1}, and the evident fact that the subset topology on any open set corresponding to a subgraph, coincides with the unit ball topology of the subgraph.
\end{rmk}
\begin{rmk}
Let us look at a case how the graph cohomology may be glued from smaller pieces at least in principle. Let $\calU$ be any open subset of $\calG$ corresponding to a subgraph $U$. Let us denote the closed subset $\calY=\calG-\calU$. Then all statements of exercises 2.3 and 2.4 on page 212 of \cite{H} apply. In particular, take $\calF=k$, we get the following long exact sequence
\begin{equation}
0\rightarrow H^0_{\calY}(\calG,k)\rightarrow H^0(\calG,k)\rightarrow H^0(\calU,k)\rightarrow H^1_{\calY}(\calG,k)\rightarrow ...
\end{equation}
Where $H^i(\calG,k)$ and $H^i(\calU,k)$ are naturally isomorphic to the usual graph cohomology, as we have seen. The additional piece $H^i_{\calY}(\calG,k)$ may be analyzed by the same flasque resolution \ref{res}. Furthermore, this cohomology with support in $\calY$ satisfies the excision and Mayer-Vietoris sequence. A tricky thing is that the combinatorial translation of such statements may not be nice or very useful in general.
\end{rmk}
\section{Some comments}\label{L}
If one regards a graph $G$ simply as a one-dimensional simplicial complex, and considers its simplicial cohomology, the procedure can again be discretized in the same way: one takes the set $\calG'=\calG_0\cup\calG_1$, and take all the star graphs \cite{K} as the subbases of topology, then in the same way, one can show that the cohomology of the constant sheaf realizes this trivial version of graph cohomology.

Let us take a look at a simple version of the Lefschetz fixed point theorem for graphs \cite{K1}, \footnote{One can also consider more elaborated versions, but here we take the simplest version for the purpose of illustration.}which states that for any automorphism $f$ of a graph $G$, one defines its Lefschetz number as
\begin{equation}
\Lambda(f)=\sum_{i=0}^{\infty}(-1)^iTr(f^*: H^i(G,k)\rightarrow H^i(G,k))
\end{equation}
Then if $\Lambda(f)$ is nonzero, $f$ has at least one fixed simplex, where $k$ is any ground field, and $H^i(G,k)$ is graph cohomology taking values in $k$. The proof of this can be reduced to the familiar case of simplicial complexes, or one shows as usual that it is a consequence of linear algebra.

From our framework, for any injective graph homomorphism $\phi: G_1\rightarrow G_2$, since it maps cliques to cliques, one has an induced continuous map of topological spaces $\calG_1\rightarrow \calG_2$, which we still denote by $\phi$ if no confusion arises. In particular, if $\phi=f$ is an automorphism of $G$,  $f$ is continuous as a map from $\calG$ to itself. On the other hand, by \ref{iso1}, the definition of $\Lambda(f)$ can also be stated using $H^i(\calG, k)$. Therefore, the above Lefschetz theorem for graphs is equivalent to a Lefschetz theorem for $f$ and the topological space $\calG$ with Lefschetz number defined by the sheaf cohomology. It looks to be an interesting question to elaborate on this observation, from the point of view of finite set topology.

Also one notes that the same theorem holds for digraphs: with regard to lemma \ref{SE}, a morphism of digraphs that has no fixed vertexes has to have zero trace in $\Omega_k$, and thus its Lefschetz number has to be zero.

\section{Path cohomology}

For all $k$, choose any integral basis of $\Omega_k(G)$ consisting of minimal elements, and let us call it $B_k$. Let $X_G$ denote the union of these basis as a set. For any path $P\in X_G$ of length $k$, we denote by $G_P$ the smallest subgraph of $G$, such that $P\in\Omega_k(G_P)$. For each $P$, we define $U_P$ inductively to be the union of $\left\{P\right\}$, and $U_Q$, where $Q$ is any element in $X_G$, that appears as a direct summand of an element in $\Omega_k(G_P)$. We define a topology $T$ on $X_G$ by claiming all $U_P$ to form a subbasis of topology. We have
\begin{lem}
$U_P$ is the smallest open subset containing $P$, and $U_{P_1}\cap U_{P_2}=\cup_{x\in U_{P_1}\cap U_{P_2}}U_x$.
\end{lem}
\begin{proof}
Check by the definitions.
\end{proof}
Thus we have
\begin{cor}
$U_P$ form a basis of topology.
\end{cor}
\begin{rmk}
The definition of $X_G$ and $U_P$ is carefully chosen, so as to take into account the subtle issues involved in the definition of the path cohomology. 
\end{rmk} 
For any $k$, we define a sheaf $\calC^k$ of abelian groups on the topological space $X_G$ as follows: for any open set $U$, one assigns the abelian group of integer valued $\bZ$-linear functions on the $\bZ$-module spanned by the set of length $k$ elements in $U$. It is obvious that $\calC^k$ is a flasque sheaf. It is straightforward to check that the sheaves $\calC^k$ form a complex of sheaves via the natural differential. Therefore, taking global sections, the cohomology of this complex of sheaves computes the path cohomology of digraphs. We have the following lemma
\begin{lem}[Poincare lemma]\label{Poincare}
$\calC^k$ is a flasque resolution of the constant sheaf.
\end{lem}

And a simpler version

\begin{lem}[Poincare lemma: baby version]
For any $P\in X_G$, we have $H^i(G_P)=0$ for all $i>0$.
\end{lem}
\begin{proof} 
The combinatorics of both lemmas are subtle, and the authors only know a combinatorial proof of the baby version lemma. On the other hand, unravelling the definitions, they evidently follow from the construction of the CW complex in section \ref{CW}, namely it follows from the proof that in the inductive process of constructing the CW complex, or in associating a cell to any minimal path, $\partial P$ gives rise to a $(k-1)$-sphere, for $P$ a minimal $k$-path. 
\end{proof}
Thus we have
\begin{thm}
The cohomology of the constant sheaf on $X_G$ is naturally isomorphic to the path cohomology of $G$.
\end{thm}

\section{Computation: a first discussion}
In this section, we are concerned with the computation of the CW complex and the cohomology, and try to get some first understanding of the complexity. We have the following:

\begin{thm}\label{complexity}
For digraphs with a uniform bound on the vertex degree, if one fixes $k$, then the time complexity of computing a basis of $\Omega_k$ consisting of minimal paths, and thus the $k$ skeleton of the CW complex, is quadratic.
\end{thm}
\begin{proof}
Let $D$ denote the uniform bound of vertex degree, and $n$ be the number of vertexes of the digraph. By lemma \ref{SE}, any minimal path has unique starting and ending vertexes. There are at most $n(n-1)$ choices of these ordered pairs of vertexes. For each such choice, there are at most $D^{k-1}$ many primitive paths of length $k$ with the given starting and ending vertexes, and once all these primitive paths are enumerated, one is left with another finite calculation to determine a rational basis of minimal paths with given starting and ending vertexes. (think of lemma \ref{min}) These basis elements combine to give a desired basis of $\Omega_k$ consisting of minimal paths.
\end{proof}

\begin{rmk}
The proof that the homotopy type of the CW complex is determined by the digraph, obviously also shows that the same is true for any $k$ skeleton.
\end{rmk}

In the following, we present a recursive scheme for computing  a basis of $\Omega_k$ consisting of minimal paths.

Take any minimal path $P$ of length $k$, and with starting vertex S and ending vertex E. As before, we single out a subset $E_1$ of the set of vertexes in the support of $P$, consisting of vertexes such that any path in the support of $P$ connecting the vertex to $E$ is of length 1. It is then clear that, for any vertex $W_1$ in $E_1$, if one groups together all primitive paths in $P$ going through $W_1$ taking signs into account, and truncates $E$ from them, then one gets a path $P'$ of length $k-1$, and furthermore $P'$ is $\partial$-invariant: the proof of this is essentially the same as that of lemma \ref{SE}. So $P'$ can be written uniquely as a sum of basis elements of length $k-1$ that has already been computed, again with all the coefficients being either 1 or $-1$, and furthermore the union of these $(k-1)$-cells corresponding to the basis elements that show up, is homeomorphic to a $(k-1)$-disk. (see previous arguments in constructing the CW complex) For $P'$, one again defines a set $E_2$ to be the subset of vertexes in the support of $P'$, such that any path connecting the vetex to $W_1$ is of length 1. One then sees that in order for $P$ to be $\partial$-invariant, it is necessary and sufficient that, for any vertex $W_2$ in $E_2$ that is not connected to $E$ by a directed edge, and any primitive path $P''$ in $P'$ that goes through $W_2$, there exists another vertex $W$ in $E_1$ such that, once one expresses the same truncation of (signed) summation of all primitive paths of $P$ that goes through $W$ in terms of the unique linear combination of $k-1$ basis elements chosen, there exists one basis element in the linear combination, that contains a primitive path given by swithcing the ending vertex of $P''$ from $W_1$ to $W$, with appropriate sign, so that boundary components of these two primitive paths given by deleting $W_1$ and $W$ cancel as desired. In this way, one finds all $\partial$-invariant paths of length $k$ between S and E, then one goes on to find the minimal ones, and a rational thus integral basis, for which efficient and straightforward algorithms exist.

\begin{rmk}
It is clear that, the above recursive scheme will be more efficient than a basic brute force algorithm following from the proof of theorem \ref{complexity}. It is a problem to carefully study the complexity of such an algorithm in more general situations.
\end{rmk}


\begin{thebibliography}{10}
\bibitem{BN} M. Baker, S. Norine, \emph{Riemann-Roch and Abel-Jacobi theory on a finite graph},  arXiv:math/0608360
\bibitem{BSS} B. Chen, S-T Yau, Y-N Yeh, \emph{Graph homotopy and Graham homotopy}, Discrete Math., 241 (2001) 153-170.
\bibitem{DH} D. Happel, \emph{Hochschild cohomology of finite dimensional algebras}, Lecture Notes in Math, Springer-Verlag 1404, 1989. 108–126.
\bibitem{F} J. Friedman, \emph{Sheaves on Graphs, Their Homological Invariants, and a Proof of the Hanna Neumann Conjecture},  arXiv:1105.0129.

\bibitem{H} R. Hartshorne, \emph{Algebraic Geometry}, GTM 52, Springer-Verlag, 1977.
\bibitem{K0} O. Knill, \emph{The Dirac operator of a graph},  arXiv:1306.2166.
\bibitem{K} O. Knill, \emph{A notion of graph homeomorphism}, arXiv:1401.2819.
\bibitem{K1} O. Knill, \emph{A Brouwer fixed point theorem for graph endomorphisms}, arXiv:1206.0782.
\bibitem{GLMY} A. Grigor'yan, Y. Lin, Y. Muranov, and S-T Yau, \emph{Homologies of path complexes and digraphs}, arXiv:1207.2834.
\bibitem{GLMY2} A. Grigor'yan, Y. Lin, Y. Muranov, and S-T Yau, \emph{Homotopy theory for digraphs}, arXiv:1407.0234.
\end{thebibliography}
\end{document}